\documentclass[12pt,twoside,reqno]{amsart}
\usepackage{amsfonts}
\usepackage{amssymb}
\usepackage[notcite,notref]{%showkeys
}

\numberwithin{equation}{section}

\newtheorem{teo}{Theorem}[section]
\newtheorem{prop}[teo]{Proposition}
\newtheorem{lem}[teo]{Lemma}
\newtheorem{cor}[teo]{Corollary}

\theoremstyle{definition}

\newtheorem{rem}[teo]{Remark}

\makeatletter
\@namedef{subjclassname@2010}{%
  \textup{2010} Mathematics Subject Classification}
\makeatother

\def\a{\alpha}
\def\b{\beta}

\def\o{\omega}
\def\R{\mathbb{R}}
\def\Z{\mathbb{Z}}
\def\N{\mathbb{N}}
\def\Q{\mathbb{Q}}
\def\d{\delta}
\def\e{\varepsilon}

\def\s{\sigma}

\def\mes{\text{mes}}
\def\pvar{V_p\,[\,V_p\,]}
\def\1var{V_1\,[\,V_1\,]}
\DeclareMathOperator*{\essup}{\text{ess}\sup}

\begin{document}

\title[]{On embeddings of spaces of bivariate functions of bounded $p$-variation}

\author{M. Lind}

\address{Department of Mathematics,
Karlstad University,
Universitetsgatan 2,
651 88 Karlstad,
SWEDEN}

\begin{abstract}
We obtain sharp estimates of the Hardy-Vitali type total $p$-variation of a function of two variables in terms of its mixed modulus of continuity in $L^p([0,1]^2)$. We also investigate various embeddings for mixed norm spaces of bivariate functions whose linear sections have bounded $p$-variation in the sense of Wiener. 
\end{abstract}

\keywords{Hardy-Vitali variation, $p$-variation, moduli of continuity, optimal constants, mixed norm spaces, embeddings}
\subjclass[2010]{Primary 26A45; Secondary 26B35, 46E35}

\maketitle

\section{Introduction}

Consider a 1-periodic function $f$ on the real line and let $1\le p<\infty$. A set $\Pi=\{x_0,x_1,...,x_n\}$ such that
$$
x_0<x_1<...<x_n=x_0+1
$$
will be called a \emph{partition}. 
Set 
\begin{equation}
\nonumber
v_p(f;\Pi)=\left(\sum_{k=0}^{n-1}|f(x_{k+1})-f(x_k)|^p\right)^{1/p}.
\end{equation}
We say that $f$ is of bounded $p$-variation (written $f\in V_p$) if
\begin{equation}
\label{Paper5WienerpvarDef}
v_p(f)=\sup_\Pi v_p(f;\Pi)<\infty,
\end{equation}
where the supremum is taken over all partitions $\Pi$. For $p=1$, this definition was given by Jordan, and for $p>1$ by Wiener \cite{Wi}. Subsequently, other extensions of the concept of bounded variation have been considered by many authors, see, e.g., \cite{Be1,MuOr1, Wa1}

We denote by $L^p([0,1]^n)~~(1\le p\le\infty)$ the class of all measurable functions $f$ on $\R^n$ that are 1-periodic in each variable and satisfy
$$
\|f\|_p=\left(\int_{[0,1]^n}|f(x)|^pdx\right)^{1/p}<\infty,\quad1\le p<\infty,
$$
or $\|f\|_\infty=\essup_{x\in[0,1]^n}|f(x)|<\infty$ for $p=\infty$.
For $f\in L^p([0,1]^n)$ and $h\in\R^n$, set
\begin{equation}
\label{Paper5DeltaIntr0}
\Delta(h)f(x)=f(x+h)-f(x).
\end{equation}
The \emph{$L^p$-modulus of continuity} of $f$ is defined as
\begin{equation}
\o(f;\d)_p=\sup_{|h|\le\d}\|\Delta(h)f\|_p,\quad0<\d\le1.
\end{equation}

In the one-dimensional case, the relationship between integral smoothness and variational properties of functions has been studied for a long time; we refer to \cite{KLi} and the references given therein. In \cite{KLi}, the following result was proved.
Let $f\in L^p([0,1])~~(1<p<\infty)$ and assume that
\begin{equation}
\label{Paper5onedimcond}
\int_0^1t^{-1/p}\o(f;t)_p\frac{dt}{t}<\infty.
\end{equation}
Then $f$ is equivalent to a continuous 1-periodic function $\bar{f}\in V_p$. Moreover, 
\begin{equation}
\label{Paper5KLiTeo1eq0}
\|f\|_\infty\le A\left[\|f\|_p+\frac{1}{pp'}\int_0^1t^{-1/p}\o(f;t)_p\frac{dt}{t}\right],
\end{equation}
and
\begin{equation}
\label{Paper5KLiTeo1eq1}
v_p(\bar{f})\le A\left[\o(f;1)_p+\frac{1}{pp'}\int_0^1t^{-1/p}\o(f;t)_p\frac{dt}{t}\right],
\end{equation}
where $A$ is an absolute constant and $p'=p/(p-1)$.

Inequalities of the type (\ref{Paper5KLiTeo1eq0}) and (\ref{Paper5KLiTeo1eq1}) were obtained much earlier by Geronimus \cite{Ge} and Terehin \cite{Te1} respectively. The novelty of the above estimates is that the constant coefficients in them have the sharp asymptotic behaviour as $p\rightarrow1$ and $p\rightarrow\infty$.

One of the main objectives of this paper is to to study the relations between integral smoothness and variational properties of functions of two variables. 

Several definitions of the concept of bounded variation in higher dimensions have been suggested, we refer to \cite{AC1} for an overview of some classical extensions. For more recent developments, see, e.g., \cite{DW1} and the references given therein.

A set $\mathcal{N}=\{(x_i,y_j):0\le i\le m,0\le j\le n\}$ of points in $\R^2$ such that
$$
x_0<x_1<...<x_m=x_0+1,\quad y_0<y_1<...<y_n=y_0+1,
$$
will be called a \emph{net}. Let $1\le p<\infty$ and let the function $f(x,y)$ be 1-periodic in both variables. For a fixed net $\mathcal{N}$, we denote
$$
\Delta f(x_i,y_j)=f(x_{i+1},y_{j+1})-f(x_{i+1},y_j)-f(x_i,y_{j+1})+f(x_i,y_j),
$$
for $0\le i\le m-1$, $0\le j\le n-1$, and 
$$
v_p^{(2)}(f;\mathcal{N})=\left(\sum_{i=0}^{m-1}\sum_{j=0}^{n-1}|\Delta f(x_i,y_j)|^p\right)^{1/p}.
$$
The space $V_p^{(2)}$ consists of all functions that satisfy
\begin{equation}
\label{Paper5VitaliDef1}
v_p^{(2)}(f)=\sup_\mathcal{N} v_p^{(2)}(f;\mathcal{N})<\infty,
\end{equation}
where the supremum is taken over all nets $\mathcal{N}$.

If $f(x,y)$ is 1-periodic in both variables and $x\in\R$ is fixed, then the $x$-section of $f$ is the 1-periodic function $f_x$ defined by
$$
f_x(y)=f(x,y),\quad y\in\R.
$$
The $y$-sections of $f$ are defined analogously. 

The space $H_p^{(2)}\subset V_p^{(2)}$ consists of all functions that, in addition to (\ref{Paper5VitaliDef1}), also satisfy the following conditions on their sections: for any $x,y\in\R$, we have $f_x,f_y\in V_p$. Observe that the class $H_p^{(2)}$ contains only bounded functions, while a function in $V_p^{(2)}$ may be unbounded.

For $p=1$, the definition (\ref{Paper5VitaliDef1}) was given by Vitali, and Hardy was the first who considered the class $H_1^{(2)}$ (see, e.g., \cite{AC1}). For $p>1$, the classes $V_p^{(2)}$ and $H_p^{(2)}$ were first defined and studied by Golubov \cite{G1}.

Let $s,t\in\R$. We shall use the following notation.
\begin{eqnarray}
\nonumber
\lefteqn{\Delta(s,t)f(x,y)=}\\
\label{Paper5DeltaIntr1}
&&f(x+s,y+t)-f(x+s,y)-f(x,y+t)+f(x,y).
\end{eqnarray}
For $1\le p<\infty$ and $f\in L^p([0,1]^2)$, the \emph{mixed $L^p$-modulus of continuity} is defined by
\begin{equation}
\nonumber
\o(f;u,v)_p=\sup_{0\le s\le u,\;0\le t\le v}\|\Delta(s,t)f\|_p.
\end{equation}

It was proved by Golubov \cite{G1} that if $f\in V_p^{(2)}~~(1\le p<\infty)$, then
\begin{equation}
\label{Paper5HardyLittIntr}
\o(f;u,v)_p\le v_p^{(2)}(f)u^{1/p}v^{1/p}.
\end{equation}
We prove (Theorem \ref{Paper5HardyLittlewood} below) that for $p=1$, the converse is also true: if $\o(f;u,v)_1=O(uv)$, then there exists $\bar{f}\in V_1^{(2)}$ such that $f=\bar{f}$ almost everywhere, and moreover,
$$
v_1^{(2)}(\bar{f})=\sup_{u,v>0}\frac{\o(f;u,v)_1}{uv}.
$$
This result is a two-dimensional analogue of a classical theorem of Hardy and Littlewood \cite[Theorem 24]{HL1}. 
For $p>1$, the condition $\o(f;u,v)_p=O(u^{1/p}v^{1/p})$ does not imply that $f$ is equivalent to some function in $V_p^{(2)}$ (this follows from corresponding result in one dimension).
 
Let $f\in L^p([0,1]^2)~~(1<p<\infty)$ and denote
$$
I_p(f)=\int_0^1\int_0^1(uv)^{-1/p}\o(f;u,v)_p\frac{du}{u}\frac{dv}{v},
$$
and 
$$
K_p(f)=\int_0^1t^{-1/p}[\o(f;t,1)_p+\o(f;1,t)_p]\frac{dt}{t}.
$$
Then we have $K_p(f)\le 4I_p(f)/p'$ (see (\ref{Paper5KpIp1}) below). One of the main results of this paper is the following: if $I_p(f)<\infty$ holds, then there exists $\bar{f}\in V_p^{(2)}$ such that $f=\bar{f}$ almost everywhere. Moreover, we have the estimate
\begin{equation}
\label{Paper5MainIntr1}
v_p^{(2)}(\bar{f})\le A\left[\o(f;1,1)_p+\frac{1}{pp'}K_p(f)+\left(\frac{1}{pp'}\right)^2I_p(f)\right],
\end{equation}
where $A$ is an absolute constant. We show that the constant coefficients in (\ref{Paper5MainIntr1}) have the optimal asymptotic behaviour as $p\rightarrow1$ or $p\rightarrow\infty$.

We also consider a two-dimensional analogue of (\ref{Paper5KLiTeo1eq0}). Potapov \cite{Pot1,Pot2,Pot3} obtained estimates of the $L^\infty$-norm of a function in terms of its mixed $L^p$-modulus of continuity (see also \cite{PST1}). However, the behaviour of the constant coefficients in these estimates were not investigated. We study this problem in Section 3 below. Observe first that for $f\in L^p([0,1]^2)~~(1<p<\infty)$, the condition $I_p(f)<\infty$ alone is not sufficient to ensure that $f\in L^\infty([0,1]^2)$. Indeed, if $f(x,y)=g(x,y)+\phi(x)$, then $I_p(f)=I_p(g)$, but $\phi$ is an arbitrary function (in particular, $\phi$ can be unbounded). However, we prove below that if we in addition to $I_p(f)<\infty$ also assume that
$$
J_p(f)=\int_0^1t^{-1/p}\o(f;t)_p\frac{dt}{t}<\infty,
$$
then $f$ is equal almost everywhere to a continuous function, and there exists an absolute constant $A>0$ such that
\begin{equation}
\label{Paper5MainIntr2}
\|f\|_\infty\le A\left[\|f\|_p+\frac{1}{pp'}J_p(f)+\left(\frac{1}{pp'}\right)^2I_p(f)\right],
\end{equation}
holds. In the same way as in the estimate (\ref{Paper5MainIntr1}), the constant coefficients of (\ref{Paper5MainIntr2}) are optimal.

Let $f$ be a given function and define the functions
$$
\varphi_1[f](x)= v_1(f_x)\quad{\rm  and}\quad\psi_1[f](y)=v_1(f_y).
$$
In \cite{AC1}, it was proved that if $f\in H_1^{(2)}$, then $\varphi_1[f],\psi_1[f]\in V_1$.
This result led us to consider the following mixed norm spaces. Let $f(x,y)$ be 1-periodic in both variables. For $1\le p<\infty$, define
$$
\varphi_p[f](x)=v_p(f_x)\quad{\rm and}\quad\psi_p[f](y)=v_p(f_y).
$$
We denote by $\pvar$ the collection of all functions $f$ such that
\begin{equation}
\label{Paper5VpVpDef}
W_p(f)=v_p(\varphi_p[f])+v_p(\psi_p[f])<\infty.
\end{equation}
The result of \cite{AC1} mentioned above states that the embedding $H_1^{(2)}\subset\1var$ holds. It is natural to consider the relation between $\pvar$ and $H_p^{(2)}$ for $p>1$. Here, the situation is different. We prove by direct constructions (see Proposition \ref{Paper5EasyProp} and Theorem \ref{Paper5VpVpTeo1}) that for $1<p<\infty$,
$$
\pvar\not\subset H_p^{(2)}\quad{\rm and}\quad H_p^{(2)}\not\subset\pvar.
$$
We also prove that if $I_p(f)<\infty$ and $J_p(f)<\infty$ holds, then $f$ is equal a.e. to a function $\bar{f}\in\pvar$ such that
$$
W_p(\bar{f})\le A\left[\o(f;1,1)_p+\frac{1}{pp'}K_p(f)+\left(\frac{1}{pp'}\right)^2I_p(f)\right],
$$
where $A$ is an absolute constant. This is a complement of inequality (\ref{Paper5MainIntr1}).

{\sc Acknowledgments.} This work was completed under the supervision of Professor V.I. Kolyada, to whom the author is very grateful.

%%%%%%%%%%%%%%%%%%%%%%%%%%%%%%%%%%%%%%%%%%%%%%%%%%%%%%%%%%%%%%%%%%%%%%%%%%%%%%%%%%%%%%%%%%%%%%%%%%%%%%%%%%%%%%%%%%%%%%%%%%%%%%%%%%%%%%%%%%%%%%%%%%%%%%%%%%%%%%%%%%%%%%%%%%%%%%

\section{Auxiliary results}

We collect some results on moduli of continuity.
Let $f\in L^p([0,1]^2)$. Then $\o(f;\d)_p$ is nondecreasing and
\begin{equation}
\label{Paper5Mod1}
\o(f;2\d)_p\le 2\o(f;\d)_p,\quad0\le\d\le1/2.
\end{equation}
Similarly, for a fixed $v\in[0,1]$, $\o(f;u,v)_p$ is nondecreasing in the first variable and
\begin{equation}
\label{Paper5Mod2}
\o(f;2u,v)_p\le2\o(f;u,v)_p,\quad0\le u\le1/2.
\end{equation}
Consequently, we have
\begin{equation}
\label{Paper5Mod3}
\frac{\o(f;u_1,v)_p}{u_1}\le 2\frac{\o(f;u_2,v)_p}{u_2},\quad 0<u_2\le u_1\le1.
\end{equation}
Similar relations hold with respect to the second variable $v$ for a fixed $u\in[0,1]$.

Let $h\in\R$, we shall use the following notations.
\begin{equation}
\label{Paper5Difference1}
\Delta_1(h)f(x,y)=f(x+h,y)-f(x,y)
\end{equation}
and
\begin{equation}
\label{Paper5Difference2}
\Delta_2(h)f(x,y)=f(x,y+h)-f(x,y).
\end{equation}
The mixed difference (\ref{Paper5DeltaIntr1}) can be written as an iterated difference
$$
\Delta(s,t)f(x,y)=\Delta_1(s)\Delta_2(t)f(x,y)=\Delta_1(s)\Delta_2(t)f(x,y).
$$
From here,
$$
\|\Delta(s,t)(\Delta_1(h)f)\|_p=\|\Delta_1(s)\Delta_1(h)\Delta_2(t)f\|_p.
$$
Applying the triangle inequality, we obtain the second estimate of the next lemma (the first inequality is proved similarly).
\begin{lem}
\label{Paper5ModulusLemma}
Let $f\in L^p([0,1]^2)~~(1\le p<\infty)$ and $h\in\R$. Then
\begin{equation}
\label{Paper5ModulusLemmaeq1}
\o(\Delta_1(h)f;\d)_p\le2\min\{\o(f;\d)_p,\o(f;h)_p\},
\end{equation}
and
\begin{equation}
\label{Paper5ModulusLemmaeq2}
\o(\Delta_1(h)f;u,v)_p\le2\min\{\o(f;u,v)_p,\o(f;h,v)_p\}.
\end{equation}
Similar estimates also hold if we consider $\Delta_2(h)f$. 
\end{lem}

Let $f\in L^p([0,1]^2)~~(1<p<\infty)$. We recall the notations
\begin{equation}
\label{Paper5Integral1}
J_p(f)=\int_0^1t^{-1/p}\o(f;t)_p\frac{dt}{t},
\end{equation}
\begin{equation}
\label{Paper5Integral2}
K_p(f)=\int_0^1t^{-1/p}[\o(f;t,1)_p+\o(f;1,t)_p]\frac{dt}{t},
\end{equation}
and
\begin{equation}
\label{Paper5Integral3}
I_p(f)=\int_0^1\int_0^1(uv)^{-1/p}\o(f;u,v)_p\frac{du}{u}\frac{dv}{v}.
\end{equation}
Let $f\in L^p([0,1]^2)~~(1<p<\infty)$, then we have
\begin{equation}
\label{Paper5KpIp1}
K_p(f)\le\frac{4}{p'}I_p(f).
\end{equation}
Indeed, by (\ref{Paper5Mod3})
\begin{eqnarray}
\nonumber
I_p(f)&=&\int_0^1u^{-1/p-1}\left(\int_0^1v^{-1/p}\frac{\o(f;u,v)_p}{v}dv\right)du\\
\nonumber
&\ge&\frac{1}{2}\int_0^1u^{-1/p-1}\o(f;u,1)_pdu\int_0^1v^{-1/p}dv.
\end{eqnarray}
Thus,
$$
\int_0^1t^{-1/p}\o(f;t,1)_p\frac{dt}{t}\le\frac{2}{p'}I_p(f).
$$
Similarly, one shows
$$
\int_0^1t^{-1/p}\o(f;1,t)_p\frac{dt}{t}\le\frac{2}{p'}I_p(f),
$$
and (\ref{Paper5KpIp1}) follows. In the same way, one demonstrates that
\begin{equation}
\label{Paper5KpIp3}
\o(f;1,1)_p\le\frac{4}{(p')^2}I_p(f).
\end{equation}

Denote by $L^p_0([0,1]^2)$ the subspace of $L^p([0,1]^2)$ that consists of functions $f$ such that
$$
\int_0^1f(x,t)dt=\int_0^1f(t,y)dt=0
$$
for a.e. $x,y\in\R$. Observe that every function $f\in L^p([0,1]^2)$ can be written as
\begin{equation}
\label{Paper5Decomposition}
f(x,y)=\bar{f}(x,y)+\phi_1(x)+\phi_2(y),\quad{\rm a.e.}\quad(x,y)\in\R^2,
\end{equation}
where $\bar{f}\in L^p_0([0,1]^2)$. Indeed, let
\begin{equation}
\label{Paper5Phi1}
\phi_1(x)=\int_0^1f(x,t)dt,
\end{equation}
\begin{equation}
\label{Paper5Phi2}
\quad\phi_2(y)=\int_0^1f(t,y)dt-\iint_{[0,1]^2} f(s,t)dsdt.
\end{equation}
Then the function
$$
\bar{f}(x,y)=f(x,y)-\phi_1(x)-\phi_2(y)
$$
belongs to $L^p_0([0,1]^2)$.

It was proved in \cite{Pot3} that if $f\in L^p_0([0,1]^2)$, then 
$$
\o(f;\d)_p\le3[\o(f;\d,1)_p+\o(f;1,\d)_p],\quad 0\le\d\le1.
$$
Whence, it follows that if $f\in L^p_0([0,1]^2)~~(1<p<\infty)$, then
\begin{equation}
\label{Paper5KpIp2}
J_p(f)\le 3K_p(f).
\end{equation}

We shall also need some results on the $L^p$-modulus of continuity for a function of one variable.
The first result is well-known (see, e.g., \cite{K1} for a proof for functions on the real line; the proof in the periodic case is the same). 
\begin{lem}
\label{Paper5ViktorLemma}
Let $f\in L^p([0,1]),~~1\le p<\infty$. Then 
\begin{equation}
\label{Paper5ViktorLemmaeq1}
\o(f;\d)_p\le\frac{3}{\d}\int_0^\d\|\Delta(t)f\|_pdt,\quad\d\in(0,1],
\end{equation}
and consequently,
\begin{equation}
\label{Paper5ViktorLemmaeq2}
\int_0^1t^{-1/p}\o(f;t)_p\frac{dt}{t}\le 3\int_0^1t^{-1/p}\|\Delta(t)f\|_p\frac{dt}{t}.
\end{equation}
\end{lem}

Let $f\in L^p([0,1])~~(1\le p<\infty)$ and denote
\begin{equation}
\label{Paper5Omegap1}
\Omega_p(f)=\left(\int_0^1\int_0^1|f(x)-f(y)|^pdxdy\right)^{1/p}.
\end{equation}
Then
\begin{equation}
\label{Paper5Omegap2}
\Omega_p(f)\le\o(f;1)_p\le2\Omega_p(f),
\end{equation}
see \cite[p.589]{KLi}.

If $f(x,y)=g(x)h(y)$, then for all $p\ge1$, there holds
\begin{equation}
\label{Paper5Sharpness3}
v_p^{(2)}(f)=v_p(g)v_p(h),
\end{equation}
and
\begin{equation}
\label{Paper5Sharpness4}
\o(f;u,v)_p=\o(g;u)_p\o(h;v)_p,\quad u,v\in[0,1].
\end{equation}

Recall that when defining the class $H_p^{(2)}$ (see the Introduction), we require in addition to (\ref{Paper5VitaliDef1}) also that the sections $f_x,f_y\in V_p$ for \emph{all} $x,y\in\R$. However, it is sufficient to assume that that there exists at least two values $x_0,y_0\in\R$ such that $f(x_0,\cdot),f(\cdot,y_0)\in V_p$. Indeed, it is easy to show that for $p\ge1$ and any $x$-section $f_x$, we have
$$
v_p(f_x)\le v_p^{(2)}(f)+v_p(f_{x_0}),
$$ 
and similar inequalities hold for the $y$-sections. 

The next result is due to Golubov \cite{Go1}.
\begin{lem}
\label{Paper5PartialDerviativeLemma}
Assume that $f\in L^1_0([0,1]^2)$ and let
$$
F(x,y)=\int_0^x\int_0^yf(s,t)dsdt.
$$
Then
\begin{equation}
\label{Paper5PartialDerivative}
v_1^{(2)}(F)=\int_0^1\int_0^1|f(x,y)|dxdy.
\end{equation}
\end{lem}
\begin{rem}
The condition $f\in L^1_0([0,1]^2)$ is imposed to assure that $F$ is 1-periodic in both variables. 
\end{rem}

We shall also need the following lemma, which is a special case of a Helly-type principle proved in \cite{Leo1}.
\begin{lem}
\label{Paper5HellyLemma}
Let $\{f_n\}$ be a sequence of functions in $H_1^{(2)}$. Assume that there exist $x_0,y_0\in\R$ and $M>0$ such that the estimate
$$
v_1^{(2)}(f_n)+v_1(f_n(\cdot,y_0))+v_1(f_n(x_0,\cdot))+|f_n(x_0,y_0)|\le M
$$
holds uniformly in $n$. Then there exists a subsequence $\{f_{n_j}\}$ that converges at every point to a function $f\in H_1^{(2)}$.
\end{lem}

%%%%%%%%%%%%%%%%%%%%%%%%%%%%%%%%%%%%%%%%%%%%%%%%%%%%%%%%%%%%%%%%%%%%%%%%%%%%%%%%%%%%%%%%%%%%%%%%%%%%%%%%%%%%%%%%%%%%%%%%%%%%%%%%%%%%%%%%%%%%%%%%%%%%%%%%%%%%%%%%%%%%%%%%%%%%

\section{Estimates of the $L^\infty$-norm}
Below, we shall use the notations (\ref{Paper5Integral1}) and (\ref{Paper5Integral3}).
\begin{teo}
\label{Paper5ConvergenceTeo}
Let $f\in L^p([0,1]^2)~~(1<p<\infty)$ and suppose that 
\begin{equation}
\label{Paper5ConvergenceTeoeq0}
J_p(f)<\infty\quad{\rm and}\quad I_p(f)<\infty.
\end{equation}
Then $f$ is equal a.e. to a continuous function and 
\begin{equation}
\label{Paper5ConvergenceTeoeq1}
\|f\|_\infty\le A\left[\|f\|_p+\frac{1}{pp'}J_p(f)+\left(\frac{1}{pp'}\right)^2I_p(f)\right],
\end{equation}
where $A$ is an absolute constant.
\end{teo}
\begin{proof}
Assume that (\ref{Paper5ConvergenceTeoeq0}) holds, we shall first prove the estimate (\ref{Paper5ConvergenceTeoeq1}). For each $x\in[0,1]$, we apply (\ref{Paper5KLiTeo1eq0}) to the $x$-section $f_x$. Using also (\ref{Paper5ViktorLemmaeq2}), we have
\begin{equation}
\label{Paper5ConvergenceTeoeq2}
\|f_x\|_\infty\le A\left[\|f_x\|_p+\frac{1}{pp'}\int_0^1v^{-1/p-1}\|\Delta(v)f_x\|_pdv\right],
\end{equation}
where $\Delta(v)f_x(y)=f(x,y+v)-f(x,y)$. Put
\begin{equation}
\label{Paper5AlfaBeta1}
\a(x)=\|f_x\|_p,\quad \b_v(x)=\|\Delta(v)f_x\|_p,
\end{equation}
and
\begin{equation}
\label{Paper5PhiDef}
\Phi(x)=\a(x)+\frac{1}{pp'}\int_0^1v^{-1/p-1}\b_v(x)dv.
\end{equation}
By (\ref{Paper5ConvergenceTeoeq2})
\begin{equation}
\label{Paper5ConvergenceTeoeq3}
\|f\|_\infty=\essup_{0\le x\le1}\|f_x\|_\infty\le A\essup_{0\le x\le1}\Phi(x).
\end{equation}
We shall estimate $\|\Phi\|_\infty$. By (\ref{Paper5KLiTeo1eq0}) and (\ref{Paper5ViktorLemmaeq2}), we have
\begin{equation}
\label{Paper5ConvergenceTeoeq4}
\|\Phi\|_\infty\le A\left[\|\Phi\|_p+\frac{1}{pp'}\int_0^1u^{-1/p-1}\|\Delta(u)\Phi\|_pdu\right],
\end{equation}
where $\Delta(u)\Phi(x)=\Phi(x+u)-\Phi(x)$. It follows easily from the definitions (\ref{Paper5AlfaBeta1}) that
\begin{equation}
\label{Paper5Delta1}
\|\a\|_p=\|f\|_p,\quad\|\Delta(u)\a\|_p\le\o(f;u)_p,
\end{equation}
and
\begin{equation}
\label{Paper5Delta2}
\|\b_v\|_p\le\o(f;v)_p,\quad\|\Delta(u)\b_v\|_p\le\o(f;u,v)_p.
\end{equation}
We estimate both terms of (\ref{Paper5ConvergenceTeoeq4}), starting with $\|\Phi\|_p$. By Minkowski's inequality and the left inequalities of (\ref{Paper5Delta1}) and (\ref{Paper5Delta2}), we get
\begin{eqnarray}
\nonumber
\|\Phi\|_p&\le&\|f\|_p+\frac{1}{pp'}\left(\int_0^1\left(\int_0^1v^{-1/p-1}\b_v(x)dv\right)^pdx\right)^{1/p}\\
\nonumber
&\le&\|f\|_p+\frac{1}{pp'}\int_0^1v^{-1/p-1}\left(\int_0^1\b_v(x)^pdx\right)^{1/p}dv\\
\label{Paper5ConvergenceTeoeq5}
&\le&\|f\|_p+\frac{1}{pp'}J_p(f).
\end{eqnarray}
We proceed to estimate $\|\Delta(u)\Phi\|_p$. Put
$$
I(x)=\int_0^1v^{-1/p-1}\b_v(x)dv.
$$
Then, by Minkowski's inequality and the right inequality of (\ref{Paper5Delta1})
\begin{eqnarray}
\nonumber
\|\Delta(u)\Phi\|_p&\le&\|\Delta(u)\a\|_p+\frac{1}{pp'}\|\Delta(u)I\|_p\\
\label{Paper5ConvergenceTeoeq6}
&\le&\o(f;u)_p+\frac{1}{pp'}\|\Delta(u)I\|_p. 
\end{eqnarray}
Further, since
$$
|I(x+u)-I(x)|\le\int_0^1v^{-1/p-1}|\b_v(x+u)-\b_v(x)|dv,
$$
we get after applying Minkowski's inequality that
\begin{eqnarray}
\nonumber
\|\Delta(u)I\|_p&\le&\left(\int_0^1\left(\int_0^1v^{-1/p-1}|\Delta(u)\b_v(x)|dv\right)^pdx\right)^{1/p}\\
\label{Paper5ConvergenceTeoeq7}
&\le&\int_0^1v^{-1/p-1}\|\Delta(u)\b_v\|_pdv.
\end{eqnarray}
By (\ref{Paper5ConvergenceTeoeq6}), (\ref{Paper5ConvergenceTeoeq7}) and the right inequality of (\ref{Paper5Delta2}), we have
\begin{equation}
\label{Paper5ConvergenceTeoeq8}
\|\Delta(u)\Phi\|_p\le\o(f;u)_p+\frac{1}{pp'}\int_0^1v^{-1/p-1}\o(f;u,v)_pdv.
\end{equation}
Now, (\ref{Paper5ConvergenceTeoeq1}) follows from (\ref{Paper5ConvergenceTeoeq3}), (\ref{Paper5ConvergenceTeoeq4}), (\ref{Paper5ConvergenceTeoeq5}) and (\ref{Paper5ConvergenceTeoeq8}).

We now prove that $f$ agrees a.e. with a continuous function. To do this, it is sufficient to show that $\o(f;\d)_\infty\rightarrow0$ as $\d\rightarrow0$. Fix $\d\in(0,1]$, then
$$
\o(f;\d)_\infty\le\sup_{0\le h\le\d}\|\Delta_1(h)f\|_\infty+\sup_{0\le h\le\d}\|\Delta_2(h)f\|_\infty,
$$
where $\Delta_1(h)f, \Delta_2(h)f$ are defined by (\ref{Paper5Difference1}) and (\ref{Paper5Difference2}) respectively.
For $h\in(0,\d]$, we have by (\ref{Paper5ConvergenceTeoeq1}) that
$$
\|\Delta_1(h)f\|_\infty\le A\left[\|\Delta_1(h)f\|_p+\frac{1}{pp'}J_p(\Delta_1(h)f)+\left(\frac{1}{pp'}\right)^2I_p(\Delta_1(h)f)\right].
$$
By using Lemma \ref{Paper5ModulusLemma}, (\ref{Paper5Mod1}) and (\ref{Paper5Mod2}), we get for any $0<h\le\d$
$$
J_p(\Delta_1(h)f)\le c\int_0^\d t^{-1/p}\o(f;t)_p\frac{dt}{t},
$$
and
$$
I_p(\Delta_1(h)f)\le c\int_0^\d\int_0^1(uv)^{-1/p}\o(f;u,v)_p\frac{dv}{v}\frac{du}{u},
$$
for some constant $c$ that is independent of $\d$. It follows that
$$
\lim_{\d\rightarrow0}(\sup_{0\le h\le\d}\|\Delta_1(h)f\|_\infty)=0.
$$
In exactly the same way, we can show that
$$
\lim_{\d\rightarrow0}(\sup_{0\le h\le\d}\|\Delta_2(h)f\|_\infty)=0.
$$
Hence, $\lim_{\d\rightarrow0}\o(f;\d)_\infty=0$. This concludes the proof.
\end{proof}
\begin{cor}
\label{Paper5Cor1}
Let $f\in L^p([0,1]^2)\;(1<p<\infty)$ and assume that $I_p(f)<\infty$. Then there exist a continuous function $g\in L^p_0([0,1]^2)$ and univariate functions $\phi_1,\phi_2$ such that
$$
f(x,y)=g(x,y)+\phi_1(x)+\phi_2(y),
$$
for a.e. $(x,y)\in\R^2$.
\end{cor}
\begin{proof}
By (\ref{Paper5Decomposition}), we have
$$
f(x,y)=\bar{f}(x,y)+\phi_1(x)+\phi_2(y)
$$
for a.e. $(x,y)\in\R^2$, where $\bar{f}\in L^p([0,1]^2)$. We shall prove that $\bar{f}$ is equal a.e. to a continuous function $g$. Clearly $I_p(\bar{f})=I_p(f)<\infty$, and since $\bar{f}\in L^p_0([0,1]^2)$, we also have $J_p(\bar{f})<\infty$, by (\ref{Paper5KpIp2}) and (\ref{Paper5KpIp1}). The result now follows from Theorem \ref{Paper5ConvergenceTeo}.
\end{proof}

%%%%%%%%%%%%%%%%%%%%%%%%%%%%%%%%%%%%%%%%%%%%%%%%%%%%%%%%%%%%%%%%%%%%%%%%%%%%%%%%%%%%%%%%%%%%%%%%%%%%%%%%%%%%%%%%%%%%%%%%%%%%%%%%%%%%%%%%%%%%%%%%%%%%%%%%%%%
\section{Estimates of the Vitali type $p$-variation}
In this section we shall consider the relationship between mixed integral smoothness and the Vitali type $p$-variation.

In the case $p=1$, we have the following theorem.
\begin{teo}
\label{Paper5HardyLittlewood}
Assume that $f\in L^1([0,1]^2)$ and that 
\begin{equation}
\nonumber
\o(f;u,v)_1=O(uv).
\end{equation}
Then there exist a function $g\in H_1^{(2)}$ and univariate functions $\phi_1,\phi_2$ such that for a.e. $(x,y)\in\R^2$,
$$
f(x,y)=g(x,y)+\phi_1(x)+\phi_2(y).
$$
Moreover,
\begin{equation}
\label{Paper5HardyLittlewoodeq0}
v_1^{(2)}(g)=\sup_{u,v>0}\frac{\o(f;u,v)_1}{uv}.
\end{equation}
\end{teo}
\begin{proof}
We may without loss of generality assume that $f\in L^1_0([0,1]^2)$.
For $n\in\N$, denote
$$
f_n(x,y)=n^2\int_0^{1/n}\int_0^{1/n}f(x+s,y+t)dsdt.
$$
We shall first prove that
\begin{equation}
\label{Paper5HardyLittlewoodeq1}
v_1^{(2)}(f_n)\le\sup_{u,v>0}\frac{\o(f;u,v)_1}{uv}.
\end{equation}
Observe that
\begin{eqnarray}
\nonumber
f_n(x,y)&=&\int_0^x\int_0^yD_1D_2f_n(s,t)dsdt-f_n(x,0)-f_n(0,y)+f_n(0,0)\\
\nonumber
&=&F_n(x,y)-f_n(x,0)-f_n(0,y)+f_n(0,0).
\end{eqnarray}
Moreover, $D_1D_2f_n(s,t)=n^2\Delta(1/n,1/n)f(s,t)$. Thus, by (\ref{Paper5PartialDerivative}),
\begin{eqnarray}
\nonumber
v_1^{(2)}(f_n)&=&v_1^{(2)}(F_n)=n^2\int_0^1\int_0^1|\Delta(1/n,1/n)f(x,y)|dxdy\\
\nonumber
&\le&\sup_{u,v>0}\frac{\o(f;u,v)_1}{uv}.
\end{eqnarray}
This proves (\ref{Paper5HardyLittlewoodeq1}).

Let $E$ be the set of Lebesgue points of $f$. Since $\R^2\setminus E$ has Lebesgue measure 0, there exist $(x_0,y_0)\in E$ such that the sections
$$
E(x_0)=\{y\in\R:(x_0,y)\in E\}\quad{\rm and}\quad E(y_0)=\{x\in\R:(x,y_0)\in E\},
$$
have full measure. That is,
\begin{equation}
\label{Paper5Case1eq1}
\mes_1(\R\setminus E(x_0))=\mes_1(\R\setminus E(y_0))=0,
\end{equation}
where $\mes_1$ denotes linear Lebesgue measure. For $n\in\N$, define now
$$
g_n(x,y)=f_n(x,y)-f_n(x,y_0)-f_n(x_0,y)+f_n(x_0,y_0).
$$
For each $n\in\N$, we have $g_n(x,y_0)=g_n(x_0,y)=0$ for all $x,y\in\R$. Thus, by (\ref{Paper5HardyLittlewoodeq1}),
\begin{eqnarray}
\nonumber
v_1^{(2)}(g_n)+v_1(g_n(\cdot,y_0))+v_1(g_n(x_0,\cdot))+|g_n(x_0,y_0)|=\\
\label{Paper5VarEstim1}
=v_1^{(2)}(g_n)=v_1^{(2)}(f_n)\le\sup_{u,v>0}\frac{\o(f;u,v)_1}{uv}.
\end{eqnarray}
By Lemma \ref{Paper5HellyLemma}, there is a subsequence $g_{n_j}$ that converges at all points to a function $g\in H_1^{(2)}$. On the other hand, by (\ref{Paper5Case1eq1}) and Lebesgue's differentiation theorem, for a.e. $(x,y)\in\R^2$ there holds
$$
g(x,y)=f(x,y)-f(x,y_0)-f(x_0,y)+f(x_0,y_0).
$$
Take $\phi_1(x)=f(x,y_0)$ and $\phi_2(y)=f(x_0,y)-f(x_0,y_0)$, then $f(x,y)=g(x,y)+\phi_1(x)+\phi_2(y)$ for a.e. $(x,y)\in\R^2$. 

We now prove (\ref{Paper5HardyLittlewoodeq0}). Since $g_{n_j}$ converges to $g$ at all points, it follows from (\ref{Paper5VarEstim1}) that
$$
v_1^{(2)}(g)\le v_1^{(2)}(g_{n_j})\le\sup_{u,v>0}\frac{\o(f;u,v)_1}{uv}.
$$
On the other hand, since $f=g$ a.e., we have for any $u,v\in[0,1]$
$$
\o(f;u,v)_1=\o(g;u,v)_1\le v_1^{(2)}(g)uv,
$$
by (\ref{Paper5HardyLittIntr}).
Whence, $\sup\o(f;u,v)_1/uv\le v_1^{(2)}(g)$. This proves (\ref{Paper5HardyLittlewoodeq0}).
\end{proof}

Recall the notations (\ref{Paper5Integral2}) and (\ref{Paper5Integral3}).
\begin{teo}
\label{Paper5MainTeo1}
Let $f\in L^p([0,1]^2)~~(1<p<\infty)$ and assume that $I_p(f)<\infty$. Then there exists a continuous function $g\in H_p^{(2)}$ and univariate functions $\phi_1,\phi_2$ such that for a.e. $(x,y)\in\R^2$, we have
\begin{equation}
\label{Paper5MainTeo1eq00}
f(x,y)=g(x,y)+\phi_1(x)+\phi_2(y).
\end{equation}
Moreover,
\begin{equation}
\label{Paper5mainTeo1eq0}
v_p^{(2)}(g)\le A\left[\o(f;1,1)_p+\frac{1}{pp'}K_p(f)+\left(\frac{1}{pp'}\right)^2I_p(f)\right],
\end{equation}
where $A$ is an absolute constant. If $f\in L^p_0([0,1]^2)$, then we may take $\phi_1=\phi_2=0$ in (\ref{Paper5MainTeo1eq00}).
\end{teo}
\begin{proof}
By Corollary \ref{Paper5Cor1}, there is a continuous function $g\in L_0^p([0,1]^2)$ such that
$$
f(x,y)=g(x,y)+\phi_1(x)+\phi_2(y)
$$
for a.e. $(x,y)\in\R^2$ (if $f\in L^p_0([0,1]^2)$, then $\phi_1=\phi_2=0$). We shall prove that $g\in H_p^{(2)}$.

Take any net
$$
\mathcal{N}=\{(x_i,y_j):0\le i\le m, 0\le j\le n\},
$$
and set
$$
g_i(y)=g(x_{i+1},y)-g(x_i,y),\quad 0\le i\le m-1.
$$
Clearly,
\begin{eqnarray}
\nonumber
v_p^{(2)}(g;\mathcal{N})&=&\left(\sum_{i=0}^{m-1}\sum_{j=0}^{n-1}|\Delta g(x_i,y_j)|^p\right)^{1/p}\\
\nonumber
&=&\left(\sum_{i=0}^{m-1}\sum_{j=0}^{n-1}|g_i(y_{j+1})-g_i(y_j)|^p\right)^{1/p}\\
\label{Paper5mainTeo1eq00}
&\le&\left(\sum_{i=0}^{m-1}v_p(g_i)^p\right)^{1/p}.
\end{eqnarray}
By (\ref{Paper5KLiTeo1eq1}) and (\ref{Paper5ViktorLemmaeq2}), we have for $0\le i\le m-1$
\begin{equation}
\label{Paper5Omega1}
v_p(g_i)\le A\left(\o(g_i;1)_p+\frac{1}{pp'}\int_0^1v^{-1/p-1}\|\Delta(v) g_i\|_pdv\right),
\end{equation}
where $\Delta(v)g_i(y)=g_i(y+v)-g_i(y)$. Set
$$
I_i=\int_0^1v^{-1/p-1}\|\Delta(v)g_i\|_pdv.
$$
By (\ref{Paper5Omega1}) and (\ref{Paper5Omegap1}), we have
\begin{equation}
\label{Paper5mainTeo1eq1}
\left(\sum_{i=0}^{m-1}v_p(g_i)^p\right)^{1/p}\le A\left[\left(\sum_{i=0}^{m-1}\Omega_p(g_i)^p\right)^{1/p}+
\frac{1}{pp'}\left(\sum_{i=0}^{m-1}I_i^p\right)^{1/p}\right].
\end{equation}
Denote $g_{y,v}(x)=g(x,y+v)-g(x,y)$. Since
\begin{equation}
\nonumber
\Omega_p(g_i)^p=\int_0^1\int_0^1|g_i(y+v)-g_i(y)|^pdydv,
\end{equation}
we have
\begin{eqnarray}
\nonumber
\sum_{i=0}^{m-1}\Omega_p(g_i)^p&=&\int_0^1\int_0^1\sum_{i=0}^{m-1}|g_i(y+v)-g_i(y)|^pdydv\\
\nonumber
&=&\int_0^1\int_0^1\sum_{i=0}^{m-1}|g_{y,v}(x_{i+1})-g_{y,v}(x_i)|^pdydv\\
\nonumber
&\le&\int_0^1\int_0^1v_p(g_{y,v})^pdydv.
\end{eqnarray}
Further, by (\ref{Paper5KLiTeo1eq1}), (\ref{Paper5ViktorLemmaeq2}) and (\ref{Paper5Omegap2}), we have
$$
v_p(g_{y,v})^p\le A\left[\Omega_p(g_{y,v})^p+\left(\frac{1}{pp'}\int_0^1t^{-1/p-1}\|\Delta(t)g_{y,v}\|_pdt\right)^p\right].
$$
Thus,
\begin{eqnarray}
\nonumber
\left(\sum_{i=0}^{m-1}\Omega_p(g_i)^p\right)^{1/p}\le A\left[\left(\int_0^1\int_0^1\Omega_p(g_{y,v})^pdydv\right)^{1/p}+\right.\\
\nonumber
+\left.\frac{1}{pp'}\left(\int_0^1\int_0^1\left(\int_0^1t^{-1/p-1}\|\Delta(t)g_{y,v}\|_pdt\right)^pdydv\right)^{1/p}\right].
\end{eqnarray}
Observe that
\begin{equation}
\label{Paper5OmegaP1}
\Omega_p(g_{y,v})^p=\int_0^1\int_0^1|\Delta(h,v)g(x,y)|^pdxdh,
\end{equation}
thus
$$
\left(\int_0^1\int_0^1\Omega_p(g_{y,v})^pdydv\right)^{1/p}\le\o(g;1,1)_p.
$$
Next, by Minkowski's inequality,
\begin{eqnarray}
\nonumber
\left(\int_0^1\int_0^1\left[\int_0^1t^{-1/p-1}\left(\int_0^1|g_{y,v}(x+t)-g_{y,v}(x)|^pdx\right)^{1/p}dt\right]^pdydv\right)^{1/p}\\
\nonumber
\le\int_0^1t^{-1/p-1}\left(\int_0^1\int_0^1\int_0^1|g_{y,v}(x+t)-g_{y,v}(x)|^pdxdydv\right)^{1/p}dt\\
\nonumber
\le\int_0^1t^{-1/p-1}\left(\int_0^1\o(g;t,v)_p^pdv\right)^{1/p}dt \le\int_0^1t^{-1/p-1}\o(g;t,1)_pdt
\end{eqnarray}
Thus, we have
\begin{equation}
\label{Paper5mainTeo1eq6}
\left(\sum_{i=0}^{m-1}\Omega_p(g_i)^p\right)^{1/p}\le A\left[\o(g;1,1)_p+\frac{1}{pp'}K_p(g)\right].
\end{equation}

Now we estimate the second term of (\ref{Paper5mainTeo1eq1}). Applying Minkowski's inequality, we obtain
\begin{equation}
\label{Paper5mainTeo1eq2}
\left(\sum_{i=0}^{m-1}I_i^p\right)^{1/p}\le\int_0^1v^{-1/p-1}\left(\sum_{i=0}^{m-1}\|\Delta(v)g_i\|_p^p\right)^{1/p}dv.
\end{equation}
Furthermore,
\begin{eqnarray}
\nonumber
\sum_{i=0}^{m-1}\|\Delta(v)g_i\|_p^p&=&\int_0^1\left(\sum_{i=0}^{m-1}|g_i(y+v)-g_i(y)|^p\right)dy\\
\label{Paper5mainTeo1eq4}
&\equiv&\int_0^1S_v(y)dy.
\end{eqnarray}
On the other hand,
\begin{eqnarray}
\nonumber
S_v(y)&=&\sum_{i=0}^{m-1}|g(x_{i+1},y+v)-g(x_i,y+v)-g(x_{i+1},y)+g(x_i,y)|^p\\
\nonumber
&=&\sum_{i=0}^{m-1}|g_{y,v}(x_{i+1})-g_{y,v}(x_i)|^p,
\end{eqnarray}
where $g_{y,v}(x)=g(x,y+v)-g(x,y)$.
Thus, by (\ref{Paper5KLiTeo1eq1}) and (\ref{Paper5ViktorLemmaeq2}), for a fixed $y\in[0,1]$, we have the following estimate
\begin{eqnarray}
\nonumber
S_v(y)\le A\left(\Omega_p(g_{y,v})+\frac{1}{pp'}\int_0^1u^{-1/p-1}\|\Delta(u)g_{y,v}\|_pdu\right)^p\\
\label{Paper5mainTeo1eq9}
\le2^pA\left[\Omega_p(g_{y,v})^p+\left(\frac{1}{pp'}\int_0^1u^{-1/p-1}\|\Delta(u)g_{y,v}\|_pdu\right)^p\right].
\end{eqnarray}
Further, by (\ref{Paper5OmegaP1}),
$$
\int_0^1\Omega_p(g_{y,v})^pdy\le\o(g;1,v)_p^p.
$$
This inequality, (\ref{Paper5mainTeo1eq9}) and Minkowski's inequality yield
\begin{eqnarray}
\nonumber
\lefteqn{\left(\int_0^1S_v(y)dy\right)^{1/p}\le A'\Biggl[\o(g;1,v)_p+\Biggr.}\\
\label{Paper5mainTeo1eq3}
&&\left.+\frac{1}{pp'}\int_0^1u^{-1/p-1}\left(\int_0^1\|\Delta(u)g_{y,v}\|_p^pdy\right)^{1/p}du\right].
\end{eqnarray}
Since
$$
\|\Delta(u)g_{y,v}\|_p^p=\int_0^1|\Delta(u,v)g(x,y)|^pdx,$$
we obtain from (\ref{Paper5mainTeo1eq4}) and (\ref{Paper5mainTeo1eq3})
\begin{eqnarray}
\nonumber
\lefteqn{\left(\sum_{i=0}^{m-1}\|\Delta(v) g_i\|_p^p\right)^{1/p}\le} \\
\nonumber
&&\le A'\left[\o(g;1,v)_p+\frac{1}{pp'}\int_0^1u^{-1/p-1}\o(g;u,v)_pdu\right].
\end{eqnarray}
Integrating this inequality with respect to $v$ and taking into account 
(\ref{Paper5mainTeo1eq2}), we have
\begin{equation}
\nonumber
\left(\sum_{i=0}^{m-1}I_i^p\right)^{1/p}\le A'\left[K_p(g)+\frac{1}{pp'}I_p(g)\right].
\end{equation}
The above inequality together with (\ref{Paper5mainTeo1eq1}) and (\ref{Paper5mainTeo1eq6}) yield
\begin{eqnarray}
\nonumber
\lefteqn{\left(\sum_{i=0}^{m-1}v_p(g_i)^p\right)^{1/p}\le}\\
\label{Paper5mainEstim}
&&\le A'\left[\o(g;1,1)_p+\frac{1}{pp'}K_p(g)+\left(\frac{1}{pp'}\right)^2I_p(g)\right].
\end{eqnarray}
The estimate (\ref{Paper5mainTeo1eq0}) follows now from (\ref{Paper5mainTeo1eq00}), (\ref{Paper5mainEstim}), and the fact that $\o(g;u,v)_p=\o(f;u,v)_p$.

To show that $g\in H_p^{(2)}$, we also need to demonstrate that there exist $x,y\in\R$ such that $g_x,g_y\in V_p$. By applying (\ref{Paper5KLiTeo1eq1}) and (\ref{Paper5ViktorLemmaeq2}) to an arbitrary $x$-section $g_x$, we get
$$
v_p(g_x)\le\left[\|g_x\|_p+\frac{1}{pp'}\int_0^1v^{-1/p-1}\|\Delta(v)g_x\|_pdv\right]=A\Phi(x).
$$
It was shown in the proof of Theorem \ref{Paper5ConvergenceTeo} that if $J_p(g)$ and $I_p(g)$ are finite, then $\Phi\in L^\infty([0,1])$. Now, since $g\in L^p_0([0,1]^2)$, we have $J_p(g)\le 12I_p(g)/p'$, by (\ref{Paper5KpIp2}) and (\ref{Paper5KpIp1}). Thus, for a.e. $x\in\R$,
$$
v_p(g_x)\le A\|\Phi\|_\infty<\infty.
$$
In the same way, we have $g_y\in V_p$ for a.e. $y\in\R$. This concludes the proof.
\end{proof}

Below we shall demonstrate that the estimate (\ref{Paper5mainTeo1eq0}) is sharp in a sense.
For this, we use the following results. Let
$$
t_n(x)=\sin 2\pi nx,
$$
for $n\in\N$. It is easy to show that we have
\begin{equation}
\label{Paper5Sharpness1}
n^{1/p}\le v_p(t_n)\le2\pi n^{1/p}
\end{equation}
and
\begin{equation}
\label{Paper5Sharpness2}
\o(t_n;\d)_p\le2\pi\min(1,n\d).
\end{equation}

\begin{rem} 
Let $1<p\le2$, by (\ref{Paper5KpIp1}) and (\ref{Paper5KpIp3}), we have
$$
\o(f;1,1)_p+\frac{1}{p'}K_p(f)\le\frac{8}{(p')^2}I_p(f)
$$
Whence, for $1<p\le2$, the estimate (\ref{Paper5mainTeo1eq0}) assumes the form
\begin{equation}
\label{Paper5EstimSmallp}
v_p^{(2)}(f)\le\frac{A}{(p')^2}I_p(f).
\end{equation}
The constant $1/(p')^2$ has the optimal order as $p\rightarrow1$. Indeed, let $f(x,y)=t_1(x)t_1(y)$, then $f\in H_p^{(2)}$ for all $p\ge1$.
By (\ref{Paper5Sharpness1}) and (\ref{Paper5Sharpness3}), we have $v_p^{(2)}(f)\ge1$ for all $p\ge1$. On the other hand, by (\ref{Paper5Sharpness2}) and (\ref{Paper5Sharpness4}), we easily get that $I_p(f)\le 4\pi^2(p')^2$ for $p>1$. 
This shows that the constant coefficient $1/(p')^2$ at the right-hand side of (\ref{Paper5EstimSmallp}) cannot be replaced with some $c_p$ such that $\varliminf_{p\rightarrow1}(p')^2c_p=0$.
\end{rem}

\begin{rem}
Let $p>2$, then $1<p'<2$ and the estimate (\ref{Paper5mainTeo1eq0}) takes the form.
\begin{equation}
\label{Paper5EstimBigp}
v_p^{(2)}(f)\le A\left[\o(f;1,1)_p+\frac{1}{p}K_p(f)+\frac{1}{p^2}I_p(f)\right].
\end{equation}
We shall prove that the first term at the right-hand side of (\ref{Paper5EstimBigp}) cannot be omitted, and that the constant coefficients of the other two terms have the optimal asymptotic behaviour as $p\rightarrow\infty$.

Take first $f(x,y)=t_1(x)t_1(y)$. As above, $v_p^{(2)}(f)\ge1$ for all $p>1$ and thus $\varliminf_{p\rightarrow\infty}v_p^{(2)}(f)\ge1$. On the other hand, by (\ref{Paper5Sharpness2}) and (\ref{Paper5Sharpness4}), we have for all $p>2$ the inequalities
$$
\frac{1}{p}K_p(f)\le\frac{16\pi^2}{p}\quad{\rm and}\quad\frac{1}{p^2}I_p(f)\le\frac{16\pi^2}{p^2},
$$
This shows that the term $\o(f;1,1)_p$ of (\ref{Paper5EstimBigp}) cannot be omitted.

We proceed to show the sharpness of the constant coefficients. For fixed but arbitrary $1<p<\infty$, let $\a_p,\b_p$ be any coefficients such that
\begin{equation}
\label{Paper5BigpEstimGen}
v_p^{(2)}(f)\le A\left[\o(f;1,1)_p+\a_pK_p(f)+\b_pI_p(f)\right],
\end{equation}
holds for some absolute constant $A$ and all (continuous) functions $f\in L^p([0,1]^2)$ with $I_p(f)<\infty$. In light of Theorem \ref{Paper5MainTeo1}, we may assume that $\a_p\le1/p$ and $\b_p\le1/p^2$. We shall prove that these decay rates are optimal, i.e., that $\varliminf_{p\rightarrow\infty}p\a_p>0$ and $\varliminf_{p\rightarrow\infty}p^2\b_p>0$.

Let $f(x,y)=t_n(x)t_1(y)$, where $n\in\N$ is fixed but arbitrary. By (\ref{Paper5Sharpness2}) and (\ref{Paper5Sharpness4}), we have
$$
\o(f;u,v)_p\le 4\pi^2v\min(nu,1).
$$
Simple calculations shows that there exists an absolute constant $A>0$ such that $K_p(f)\le Apn^{1/p}$ and $I_p(f)\le Apn^{1/p}$.
On the other hand, by (\ref{Paper5Sharpness1}) and (\ref{Paper5Sharpness3}), we have $v_p^{(2)}(f)\ge n^{1/p}$.
Putting these estimates into (\ref{Paper5BigpEstimGen}) and taking into consideration that $\b_p\le1/p^2$ yield that for all $p>2$ and all $n\in\N$, we have 
$$
n^{1/p}\le A\left[1+\left(p\a_p+\frac{1}{p}\right)n^{1/p}\right],
$$
where $A$ is an absolute constant. Assume that $\varliminf_{p\rightarrow\infty}p\a_p=0$. Then, given any $\e>0$, we may choose $r=r(\e)$ such that for all $n\in\N$, there holds
$$
n^{1/r}\le A(1+\e n^{1/r}).
$$
In particular, take $\e=1/(2A)$ and choose subsequently $n\in\N$ large enough to have $n^{1/r}>2A$. This gives the contradiction
$$
n^{1/r}\le A\left(1+\frac{n^{1/r}}{2A}\right)<n^{1/r}.
$$
Whence, $\varliminf_{p\rightarrow\infty}p\a_p>0$. To show that $\varliminf_{p\rightarrow\infty}p^2\b_p>0$, take $f(x,y)=t_n(x)t_n(y)$, where $n\in\N$ is fixed but arbitrary. As above, we have
$$
\o(f;u,v)_p\le 4\pi^2\min(nv,1)\min(nu,1).
$$
Then there exists an absolute constant $A>0$ such that
$K_p(f)\le Apn^{1/p}$ and $I_p(f)\le Ap^2n^{2/p}$.
On the other hand, $v_p^{(2)}(f)\ge n^{2/p}$. Putting these estimates into (\ref{Paper5BigpEstimGen}) yields that for all $n\in\N$ and $p>2$,
$$
n^{2/p}\le A[1+p\a_pn^{1/p}+p^2\b_pn^{2/p}]
$$
where $A>0$ is an absolute constant. Dividing by $n^{1/p}$ and taking into consideration that $p\a_p\le1$, we see that 
$$
n^{1/p}\le A[2+p^2\b_pn^{1/p}],
$$
for all $p>2$ and all $n\in\N$. From here, we can give a proof by contradiction of the inequality $\varliminf_{p\rightarrow\infty}p^2\b_p>0$, as above.
\end{rem}

\begin{rem}
We shall consider trigonometric polynomials of two variables and degree $(n,m)$:
\begin{eqnarray}
\nonumber
T_{n,m}(x,y)&=&\sum_{j=0}^n\sum_{k=0}^m[a_{j,k}\cos2\pi jx\cos2\pi ky+b_{j,k}\cos2\pi jx\sin2\pi ky\\
\label{Paper5TrigPol}
&+&c_{j,k}\sin2\pi jx\cos2\pi ky+d_{j,k}\sin2\pi jx\cos2\pi ky].
\end{eqnarray}
Oskolkov \cite{Osk1} proved that for any trigonometric polynomial (\ref{Paper5TrigPol}) of degree $(n,m)$ and any $1\le p<\infty$, there holds
\begin{equation}
\label{Paper5Oskolkov}
v_p^{(2)}(T_{n,m})\le A(nm)^{1/p}\|T_{n,m}\|_p,
\end{equation}
where $A$ is an absolute constant. We can obtain (\ref{Paper5Oskolkov}) directly from (\ref{Paper5mainTeo1eq0}).
Indeed, take any trigonometric polynomial $T$ of degree $(n,m)$. The estimate
\begin{equation}
\label{Paper5Osk2}
\o(T;u,v)_p\le\min(uv\|D_1D_2T\|_p,4\|T\|_p),\quad u,v\in[0,1],
\end{equation}
is immediate.
By using (\ref{Paper5Osk2}), we get
\begin{eqnarray}
\nonumber
K_p(T)&\le&2\|D_1D_2T\|_p\int_0^{1/nm}t^{-1/p}dt+4\|T\|_p\int_{1/nm}^1t^{-1/p-1}dt\\
\label{Paper5Osk11}
&\le&2p'(nm)^{1/p-1}\|D_1D_2T\|_p+4p(nm)^{1/p}\|T\|_p.
\end{eqnarray}
It is a simple consequence of Bernstein's inequality (see \cite[p. 97]{DL}) that
\begin{equation}
\label{Paper5Bernstein}
\|D_1D_2T\|_p\le 4\pi^2nm\|T\|_p.
\end{equation}
By (\ref{Paper5Osk11}) and (\ref{Paper5Bernstein}), we get
\begin{equation}
\label{Paper5Osk3}
K_p(T)\le 12\pi^2 pp'(nm)^{1/p}\|T\|_p.
\end{equation}
Similarly, by (\ref{Paper5Osk2}),
\begin{eqnarray}
\nonumber
I_p(T)&\le&\|D_1D_2T\|_p\int_0^{1/n}\int_0^{1/m}u^{1/p}v^{1/p}dvdu\\
\nonumber
&+&4\|T\|_p\int_{1/n}^1\int_{1/m}^1(uv)^{-1/p-1}dvdu\\
\nonumber
&\le& (p')^2(nm)^{1/p-1}\|D_1D_2T\|_p+4p^2(nm)^{1/p}\|T\|_p.
\end{eqnarray}
By the above estimate and (\ref{Paper5Bernstein}), we have
\begin{equation}
\label{Paper5Osk4}
I_p(T)\le 8\pi^2(pp')^2(nm)^{1/p}\|T\|_p.
\end{equation}
Now, (\ref{Paper5Oskolkov}) is derived from (\ref{Paper5mainTeo1eq0}), the estimate $\o(T;1,1)_p\le 4\|T\|_p$, (\ref{Paper5Osk3}) and (\ref{Paper5Osk4}).
\end{rem}

%%%%%%%%%%%%%%%%%%%%%%%%%%%%%%%%%%%%%%%%%%%%%%%%%%%%%%%%%%%%%%%%%%%%%%%%%%%%%%%%%%%%%%%%%%%%%%%%%%%%%%%%%%%%%%%%%%%%%%%%%%%%%%%%%%%%%%%%%%%%%%%%%%%%%%%%%%%%%%%%%%%%%%%%%%%%%%

\section{The mixed norm space $\pvar$}
We first give some terminology from the theory of mixed norm spaces.
Let $X,Y$ be spaces of functions defined on the real line, with norms $\|\cdot\|_X$ and $\|\cdot\|_Y$. A function $f(x,y)$ is said to belong to the symmetric mixed norm space $X\,[\,Y\,]$ if the functions
$$
x\mapsto\|f_x\|_Y\quad{\rm and}\quad y\mapsto\|f_y\|_Y
$$
both belong to the space $X$ (recall that $f_x,f_y$ denote sections).

We shall consider the mixed norm spaces 
$\pvar~~(1\le p<\infty)$, and their relation to other classes, especially $H_p^{(2)}$ (see the Introduction for the definitions).

We observe that mixed norm spaces defined in terms of variation have been considered earlier. For example, Tonelli introduced and studied the space $L^1\,[\,V_1\,]$ in connection with problems of surface area (see \cite{AC1}). Another example is \cite{Go1}, where the space $L^\infty\,[\,V_1\,]$ was shown to be relevant in the study of summability of double Walsh-Fourier series.

We first remark that $\pvar$ is not a vector space for any $1\le p<\infty$.
\begin{prop}
There are two functions $f$ and $g$ such that for any $1\le p<\infty$, we have $f,g\in\pvar$ but $(f+g)\notin\pvar$.
\end{prop}
\begin{proof}
Let $f,g$ be functions that are 1-periodic in each variable, and defined as follows on $[0,1]^2$. Let $f(x,y)=1$ if $y=x$ and $f(x,y)=0$ otherwise. Set $g(x,y)=1$ if $y=x$ and $x\notin\Q$, $g(x,y)=-1$ if $y=x$ and $x\in\Q$ and $g(x,y)=0$ otherwise. Then it is easy to see that for any $x,y\in[0,1]$, we have
$$
\varphi_p[f](x)=v_p(f_x)=2^{1/p},\quad\psi_p[f](y)=v_p(f_y)=2^{1/p}.
$$
Since $\varphi_p[f],\psi_p[f]$ are constant, they are of bounded $p$-variation, that is, $f\in\pvar$. In the same way, we have $g\in\pvar$.
On the other hand, 
$$
(f+g)(x,y)=\left\{\begin{array}{ll}
		2 & \mbox{if }\; y=x\;{\rm and}\; x\notin\Q,\\
		0&\mbox{if }\; y=x\;{\rm and }\; x\in\Q,\\
		0&\mbox{otherwise}.
	\end{array}
\right.
$$
Then $\varphi_p[f+g](x)=2^{1+1/p}$ if $x\notin\Q$ and $\varphi_p[f+g](x)=0$ for $x\in\Q$. Clearly $\varphi_p[f+g]\notin V_p$.
\end{proof}

We shall consider relations between $\pvar$ and $H_p^{(2)}$ for $p\ge1$. As mentioned in the Introduction, for $p=1$, we have the strict inclusion $H_1^{(2)}\subset \1var$ (see \cite{AC1}). For $p>1$, no such embeddings hold. In fact, we shall prove that
\begin{equation}
\label{Paper5VpVpHp2rel}
\pvar\not\subset H_p^{(2)}\quad{\rm and}\quad H_p^{(2)}\not\subset\pvar\quad(p>1).
\end{equation}
The first relation is almost obvious.
\begin{prop}
\label{Paper5EasyProp}
Let $1\le p<\infty$, then there is a function $f\in\pvar$ such that $f\notin H_p^{(2)}$.
\end{prop}
\begin{proof}
Define $f$ on $(0,1]^2$
$$
f(x,y)=\left\{\begin{array}{ll}
		1 & \mbox{if }\; 0<x\le y\le1\\
		0&\mbox{if }\; 0<y<x\le1
	\end{array}
\right.
$$
and extend to the whole plane by periodicity. It is clear that $v_p(f_x)=v_p(f_y)=2^{1/p}$ for all $x,y$. Thus, $f\in\pvar$ for $1\le p<\infty$.

On the other hand, fix $n\in\N$ and let $\mathcal{N}_n=\{(x_i,y_j)\}$, where
$$
x_i=\frac{i}{n}\quad{\rm and}\quad y_j=\frac{j+1/2}{n},\quad 0\le i,j\le n.
$$
Then
$$
|\Delta f(x_i,y_i)|^p=1
$$
for $0\le i\le n-1$, whence, $v_p^{(2)}(f;\mathcal{N}_n)\ge n^{1/p}$. Thus, $f\notin H_p^{(2)}$.
\end{proof}

We proceed to show the second relation of (\ref{Paper5VpVpHp2rel}). We shall use the function
$$
\phi(x)=\inf_{k\in\Z}|x-k|.
$$
For each $n\in\N$, denote $\phi_n(x)=\phi(nx)$. Then,
\begin{equation}
\label{Paper5PhiScale1}
v_p(\phi_n)=2^{1/p-1}n^{1/p}.
\end{equation}

Define 
\begin{equation}
\label{Paper5gnDef}
g_n(x)=\phi(2^nx-1)\chi_{[0,1]}(2^nx-1)\quad{\rm for}\quad x\in[0,1],
\end{equation}
and extend $g_n$ to a 1-periodic function. Set also
\begin{equation}
\label{Paper5funktionDef}
f(x,y)=\sum_{n=1}^\infty2^{-n/p}g_n(x)\phi(2^ny).
\end{equation}
\begin{lem}
\label{Paper5VpVpLemma1}
Let $y',y''\in\R$ and
$$
g(x)=f(x,y'')-f(x,y').
$$
Then
\begin{equation}
\label{Paper5VpVpLemma1eq1}
v_p(g)\le 2^{1/p}\left(\sum_{n=1}^\infty2^{-n}|\phi(2^ny'')-\phi(2^ny')|^p\right)^{1/p}.
\end{equation}
\end{lem}
\begin{proof}
For $n\in\N$, set $\a_n=2^{-n/p}[\phi(2^ny'')-\phi(2^ny')]$, then
$$
g(x)=\sum_{n=1}^\infty\a_ng_n(x),
$$
where $g_n$ is given by (\ref{Paper5gnDef}). Set also $I_n=[2^{-n},2^{-n+1}]$ for $n\in\N$.

Let $\s'=\{n\in\N:\a_n\ge0\}$, $\s''=\N\setminus\s'$ and
$$
h_1(x)=\sum_{n\in\s'}\a_ng_n(x)\quad{\rm and}\quad h_2(x)=\sum_{n\in\s''}\a_ng_n(x).
$$
Then $h_1$ is a non-negative continuous function. Moreover, we have
$$
h_1(x)=0\quad{\rm if}\quad x\notin\bigcup_{n\in\s'}I_n,
$$
and $h_1$ is piecewise linear on each interval $I_n~~(n\in\s')$.
It is not difficult to show that the variational sum $v_p(h_1;\Pi)$ is maximal when $\Pi$ consists of the endpoints and midpoints of the intervals $I_n$ ($n\in\s'$). Thus,
$$
v_p(h_1)\le2^{1/p-1}\left(\sum_{n\in\s'}|\a_n|^p\right)^{1/p}.
$$
A similar inequality holds for $v_p(h_2)$. Hence,
$$
v_p(g)\le v_p(h_1)+v_p(h_2)\le 2^{1/p}\left(\sum_{n=1}^\infty|\a_n|^p\right)^{1/p}.
$$
\end{proof}

\begin{teo}
\label{Paper5VpVpTeo1}
Let $1<p<\infty$, then the function $f$ defined by (\ref{Paper5funktionDef}) satisfies $f\in H_p^{(2)}$ and $f\notin\pvar$.
\end{teo}
\begin{proof}
We first prove that $f\in V_p^{(2)}$. Fix any net 
$$
\mathcal{N}=\{(x_i,y_j):0\le i\le m,0\le j\le n\}.
$$
For each $j\in\{0,1,...,n-1\}$, denote
$$
g_j(x)=f(x,y_{j+1})-f(x,y_j).
$$
Since 
$$
\Delta f(x_i,y_j)=g_j(x_{i+1})-g_j(x_i),
$$
we get
$$
\sum_{i=0}^{m-1}|\Delta f(x_i,y_j)|^p=\sum_{i=0}^{m-1}|g_j(x_{i+1})-g_j(x_i)|^p\le v_p(g_j)^p.
$$
By Lemma \ref{Paper5VpVpLemma1}, we have
$$
v_p(g_j)^p\le2\sum_{k=1}^\infty2^{-k}|\phi(2^ky_{j+1})-\phi(2^ky_j)|^p,
$$
and thus,
$$
v_p^{(2)}(f;\mathcal{N})^p\le2\sum_{j=0}^{n-1}\sum_{k=1}^\infty2^{-k}|\phi(2^ky_{j+1})-\phi(2^ky_j)|^p.
$$
Set $\s_l=\{j:2^{-l-1}<y_{j+1}-y_j\le2^{-l}\}$ for $l\ge0$. Subdividing the above sum, we get
\begin{equation}
\label{Paper5VpVpTeo1eq1}
v_p^{(2)}(f;\mathcal{N})^p\le 2\sum_{l=0}^\infty\sum_{j\in\s_l}\sum_{k=1}^\infty2^{-k}|\phi(2^ky_{j+1})-\phi(2^ky_j)|^p.
\end{equation}
We shall estimate the right-hand side of (\ref{Paper5VpVpTeo1eq1}). Observe that
\begin{equation}
\label{Paper5VpVpTeo1eq2}
|\phi(2^ky_{j+1})-\phi(2^ky_j)|\le\min(1,2^k(y_{j+1}-y_j)).
\end{equation}
Indeed
$$
|\phi(2^ky_{j+1})-\phi(2^ky_j)|\le\|\phi\|_\infty=1/2, 
$$
and, at the same time,
$$
|\phi(2^ky_{j+1})-\phi(2^ky_j)|\le2^k(y_{j+1}-y_j)\|\phi'\|_\infty=2^k(y_{j+1}-y_j).
$$
Fix $l\ge0$ and let $j\in\s_l$. By (\ref{Paper5VpVpTeo1eq2}), we have
\begin{eqnarray}
\nonumber
\sum_{k=1}^\infty2^{-k}|\phi(2^ky_{j+1})-\phi(2^ky_j)|^p&\le&\sum_{k=1}^\infty2^{-k}\min(1,2^{k-l})^p\\
\nonumber
&\le&2^{-lp}\sum_{k=1}^l2^{k(p-1)}+\sum_{k=l+1}^\infty2^{-k}.
\end{eqnarray}
Since $p>1$, it follows that there is a constant $c_p>0$ such that
$$
\sum_{k=1}^\infty2^{-k}|\phi(2^ky_{j+1})-\phi(2^ky_j)|^p\le c_p2^{-l},
$$
for all $j\in\s_l$. Consequently, for $l\ge0$, there holds
\begin{equation}
\label{Paper5VpVpTeo1eq3}
\sum_{j\in\s_l}\sum_{k=1}^\infty2^{-k}|\phi(2^ky_{j+1})-\phi(2^ky_j)|^p\le c_p2^{-l}|\s_l|,
\end{equation}
where $|\s_l|$ denotes the cardinality of the finite set $\s_l$. To sum up, by (\ref{Paper5VpVpTeo1eq1}) and (\ref{Paper5VpVpTeo1eq3}), we have
$$
v_p^{(2)}(f;\mathcal{N})^p\le c_p\sum_{l=0}^\infty2^{-l}|\s_l|\le2c_p\sum_{l=0}^\infty\sum_{j\in\s_l}(y_{j+1}-y_j)=2c_p.
$$
Thus, $f\in V_p^{(2)}$. To show that $f\in H_p^{(2)}$, we must also show the existence of $x_0,y_0\in\R$ such that $f_{x_0},f_{y_0}\in V_p$. For all $x\in\R$ we have $f(x,0)=0$ and thus $f(\cdot,0)\in V_p$. Similarly, $f(1,y)=0$ for all $y\in\R$, so $f(1,\cdot)\in V_p$. Thus, $f\in H_p^{(2)}$. Show that $f\notin\pvar$. First, we observe that $g_n(2^{-k})=0\;(n,k\in\N)$. Thus, $v_p(f_x)=0$ for $x=2^{-k}\;(k\in\N)$. On the other hand, if $x=(2^{-k-1}+2^{-k})/2\;(k\in\N)$, then
$$
f_x(y)=2^{-k/p-2}\phi(2^ky),
$$
and by (\ref{Paper5PhiScale1}), we have
$$
v_p(f_x)=2^{-k/p-2}v_p(\phi_{2^k})=2^{1/p-3}.
$$
Clearly, $\varphi_p[f]\notin V_p$, and this shows that $f\notin\pvar$.
\end{proof}
\begin{rem}
One can show that
$$
H_p^{(2)}\subset L^\infty\,[\,V_p\,],\quad 1\le p<\infty.
$$
Moreover, for $p>1$, we cannot replace the exterior $L^\infty$-norm with any stronger $V_q$-norm for $1\le q<\infty$. Indeed, the construction used to prove Theorem \ref{Paper5VpVpTeo1} actually shows that for $1<p<\infty$ and all $q\ge1$, we have
$$
H_p^{(2)}\not\subset V_q\,[\,V_p\,].
$$
\end{rem}

Finally, we consider the relationship between mixed smoothness and the class $\pvar$. The next lemma will be useful in our investigation.
\begin{lem}
\label{Paper5LemmaDiff}
Let $1\le p<\infty$ and the function $f$ be 1-periodic in both variables.  Let $x',x''\in\R$ be fixed but arbitrary and assume that the $x$-sections $f_{x''},f_{x'}$ belong to $V_p$. Then
\begin{equation}
\label{Paper5VpVpLemEq0}
|v_p(f_{x''})-v_p(f_{x'})|\le 2v_p(g),
\end{equation}
where $g(y)=f(x'',y)-f(x',y)$.
\end{lem}
\begin{proof}
Denote $\varphi(x)=v_p(f_x)$ and
$$
d=|\varphi(x'')-\varphi(x')|.
$$
Observe that for any $x\in\R$ and any partition $\Gamma=\{y_0,y_1,...,y_n\}$, we have
\begin{equation}
\label{Paper5VpVpLemEq1}
\varphi(x)\ge\left(\sum_{j=0}^{n-1}|f(x,y_{j+1})-f(x,y_j)|^p\right)^{1/p}=S(x;\Gamma).
\end{equation}
Assume first that $d>0$ and that $\varphi(x'')>\varphi(x')$. By (\ref{Paper5VpVpLemEq1}), we have
\begin{equation}
\label{Paper5VpVpLemEq2}
0<d\le\varphi(x'')-S(x';\Gamma)
\end{equation}
for any partition $\Gamma$. Choose a partition $\Gamma=\{y_j:0\le j\le n\}$ such that
\begin{equation}
\label{Paper5VpVpLemEq3}
\varphi(x'')-\frac{d}{2}<S(x'';\Gamma).
\end{equation}
By (\ref{Paper5VpVpLemEq2}) and (\ref{Paper5VpVpLemEq3})
$$
d\le\varphi(x'')-S(x';\Gamma)\le S(x'';\Gamma)-S(x';\Gamma)+\frac{d}{2}.
$$
Whence, 
\begin{equation}
\label{Paper5VpVpLemEq4}
d\le 2|S(x'';\Gamma)-S(x';\Gamma)|.
\end{equation}
A simple modification of the above argument yields (\ref{Paper5VpVpLemEq4}) in the case when $\varphi(x'')<\varphi(x')$ as well.

Now, by Minkowski's inequality, we have
\begin{eqnarray}
\nonumber
\lefteqn{|S(x'';\Gamma)-S(x';\Gamma)|\le}\\
\nonumber
&&\le\left(\sum_{j=0}^{n-1}|f(x'',y_{j+1})-f(x',y_{j+1})-f(x'',y_j)+f(x',y_j)|^p\right)^{1/p}\\
\nonumber
&&\le v_p(g).
\end{eqnarray}
Thus, by (\ref{Paper5VpVpLemEq4}), we have $d\le2v_p(g)$.
\end{proof}

Recall the notations (\ref{Paper5Integral2}) and (\ref{Paper5Integral3}).
\begin{teo}
\label{Paper5VpVpTeo2}
Assume that $f\in L^p([0,1]^2)~~(1<p<\infty)$ and that $I_p(f)$ and $J_p(f)$ are finite. Then there exists a continuous function $\bar{f}\in\pvar$ such that $f=\bar{f}$ a.e., and
\begin{equation}
\label{Paper5WpEstim}
W_p(\bar{f})\le A\left[\o(f;1,1)_p+\frac{1}{pp'}K_p(f)+\left(\frac{1}{pp'}\right)^2I_p(f)\right],
\end{equation}
where $A$ is an absolute constant and $W_p(g)$ is given by (\ref{Paper5VpVpDef}).
\end{teo}
\begin{proof}
By Theorem \ref{Paper5MainTeo1}, the condition $I_p(f)<\infty$ implies that 
\begin{equation}
\label{Paper5VpVpTeo2eq1}
f(x,y)=g(x,y)+\phi_1(x)+\phi_2(y),
\end{equation}
for a.e. $(x,y)\in\R^2$, where $g\in H_p^{(2)}$ and $\phi_1,\phi_2$ are given by (\ref{Paper5Phi1}) and (\ref{Paper5Phi2}). Further, it follows directly from the definitions and Minkowski's inequality that $\o(\phi_i;\d)_p\le\o(f;\d)_p$ for $i=1,2$. Then, the condition $J_p(f)<\infty$ implies that there exist continuous functions $\bar{\phi}_1,\bar{\phi}_2\in V_p$ such that
$$
\bar{\phi}_1(x)=\phi_1(x),\quad\bar{\phi}_2(y)=\phi_2(y),
$$
for a.e. $x,y\in\R$ (see Introduction). We now define
$$
\bar{f}(x,y)=g(x,y)+\bar{\phi}_1(x)+\bar{\phi}_2(y),
$$
then $\bar{f}(x,y)=f(x,y)$ for a.e. $(x,y)\in\R^2$. Since $g\in H_p^{(2)}$, all sections $g_x,g_y\in V_p$, and since $\bar{\phi}_i\in V_p~~i=1,2$, it follows that $\bar{f}_x,\bar{f}_y\in V_p$ for all $x,y\in\R$.
Whence, the functions $\varphi(x)=v_p(\bar{f}_x)$ and $\psi(y)=v_p(\bar{f}_y)$ are well-defined and finite. We shall prove that both $v_p(\varphi)$ and $v_p(\psi)$ are estimated by the right-hand side of (\ref{Paper5WpEstim}). We do this only for $v_p(\varphi)$, the argument is the same for $v_p(\psi)$.

Let $\Pi=\{x_0,x_1,...,x_m\}$ be an arbitrary partition. Fix $0\le i\le m-1$, by Lemma \ref{Paper5LemmaDiff}, we have 
$$
|\varphi(x_{i+1})-\varphi(x_i)|\le2v_p(\bar{f}_i),
$$
where $\bar{f}_i(y)=\bar{f}(x_{i+1},y)-\bar{f}(x_i,y)$. At the same time, we have 
$$
\bar{f}_i(y)=g(x_{i+1},y)-g(x_i,y)+\phi_1(x_{i+1})-\phi_1(x_i),
$$
whence $v_p(\bar{f}_i)=v_p(g_i)$, where $g(x_{i+1},y)-g(x_i,y)$. Thus,
$$
v_p(\varphi;\Pi)=\left(\sum_{i=0}^{m-1}|\varphi(x_{i+1})-\varphi(x_i)|^p\right)^{1/p}\le2\left(\sum_{i=0}^{m-1}v_p(g_i)^p\right)^{1/p}.
$$
By this estimate, (\ref{Paper5mainEstim}) and the fact that $\o(g;u,v)_p=\o(f;u,v)_p$, we obtain
$$
v_p(\varphi)\le A\left[\o(f;1,1)_p+\frac{1}{pp'}K_p(f)+\left(\frac{1}{pp'}\right)^2I_p(f)\right].
$$
\end{proof}

\begin{rem}
The condition $J_p(f)<\infty$ of Theorem \ref{Paper5VpVpTeo2} cannot be omitted. Indeed, in general, the condition $I_p(f)<\infty$ does not even imply that $f\in L^\infty([0,1]^2)$ (see the Introduction). On the other hand, $J_p(f)$ does not appear at the right-hand side of the estimate (\ref{Paper5WpEstim}). To clarify this, observe that the condition $J_p(f)<\infty$ is only used to assure that the functions $\phi_1,\phi_2$ of (\ref{Paper5VpVpTeo2eq1}) agree a.e. with functions $\bar{\phi}_1,\bar{\phi}_2\in V_p$. This is necessary in order for $\varphi(x)=v_p(\bar{f}_x)$ and $\psi(y)=v_p(\bar{f}_y)$ to be finite for all $x,y\in\R$.
However, our reasonings show that $\bar{\phi}_1,\bar{\phi}_2$ do not contribute to the estimate of $W_p(\bar{f})$ (we do not enter into details).
\end{rem}

\end{document}